\documentclass{amsart}
\usepackage{amsfonts,amssymb,amscd,amsmath,enumerate,verbatim,calc}

\newcommand{\CM}{Cohen-Macaulay}

\newcommand{\wrt}{with respect to}
\newcommand{\B}{\mathcal{B} }

\newcommand{\n}{\mathfrak{n} }
\newcommand{\m}{\mathfrak{m} }
\newcommand{\M}{\mathfrak{M} }

\newcommand{\R}{\mathcal{R} }
\newcommand{\Z}{\mathbb{Z} }
\newcommand{\Sc}{\mathcal{R} }
\newcommand{\rt}{\rightarrow}
\newcommand{\xar}{\longrightarrow}
\newcommand{\ov}{\overline}

\newcommand{\wt}{\widetilde }

\newcommand{\depth}{\operatorname{depth}}
\newcommand{\coker}{\operatorname{coker}}
\newcommand{\ann}{\operatorname{ann}}

\theoremstyle{plain}

\newtheorem{theorem}{Theorem}[section]
\newtheorem{corollary}[theorem]{Corollary}

\newtheorem{proposition}[theorem]{Proposition}

\theoremstyle{definition}

\newtheorem{s}[theorem]{}

\theoremstyle{remark}

\begin{document}

\title{On Coefficient  ideals}
\author{Tony~J.~Puthenpurakal}
\date{\today}
\address{Department of Mathematics, IIT Bombay, Powai, Mumbai 400 076}

\email{tputhen@math.iitb.ac.in}
 \begin{abstract}
Let $(A,\mathfrak{m})$ be a  Cohen-Macaulay local ring  of dimension $d \geq 2$ with infinite residue field and let $I$ be an $\m$primary ideal. Let
 For $0 \leq i \leq d$ let $I_i$ be the $i^{th}$-coefficient ideal of $I$.  Also let $\widetilde{I} = I_d$ denote the Ratliff-Rush closure of $A$. Let $G = G_I(A)$ be the associated graded ring of $I$.
We show that if $\dim H^j_{G_+}(G)^\vee \leq j -1$ for $1 \leq j \leq i \leq d-1$ then
$(I^n)_{d-i}  = \widetilde{I^n}$ for all $n \geq 1$. In particular if $G$ is generalized  Cohen-Macaulay  then $(I^n)_1 = \widetilde{I^n}$ for all $n \geq 1$.
As a consequence we get that if $A$ is an analytically unramified domain with  $G$  generalized  Cohen-Macaulay, then  the $S_2$-ification of the Rees algebra $ A[It]$ is
$\bigoplus_{n \geq 0} \widetilde{I^n}$.
\end{abstract}
 \maketitle
\section{introduction}
Let $(A,\m)$ be a Noetherian local ring of dimension $d \geq 1$ and let $I$ be an $\m$-primary ideal.
If $M$ is an $A$-module let $\lambda(M)$ denote its length.
 Let $P_I(z) \in \mathbb{Q}[z]$ be the Hilbert-Samuel polynomial of $I$; i.e., $P_I(n) = \lambda(A/I^{n+1})$ for all $n \gg 0$.
 Write
 \[
 P_I(z) = \sum_{i = 0}^{d}(-1)^ie_i(I)\binom{z + d - i}{d-i}.
 \]
The integers $e_i(I)$ is called the $i^{th}$-Hilbert coefficient of $I$. The number $e_0(I)$ is called the multiplicity  of $A$ \wrt \ $I$.
Let $\ov{I}$ denote the integral closure of $I$
For $0 \leq i \leq d$ set
$$E_i(I) = \{ J \mid  J \supset I \ \text{and} \ e_j(J) = e_j(I) \  \   \text{for} \  0 \leq j \leq i \}. $$
Now assume $A$ is quasi-unmixed with infinite residue field. By a work of Rees \cite{Rees}, $E_0(I)$ has a unique maximal element $\ov{I}$. Shah \cite{S}  proved that
each $E_i$ has a unique maximal element $I_i$ which is called the $i^{th}$ coefficient ideal of $I$. So we have a chain of ideals
\[
 I \subseteq I_d \subseteq I_{d-1} \subseteq I_{d-2} \subseteq \cdots \subseteq I_1 \subseteq I_0 = \ov{I}.
\]
If $\depth A > 0$ then $I_d = \widetilde{I}$ the Ratliff-Rush closure of $I$. Recall
\[
 \widetilde{I} = \bigcup_{ n \geq 1}(I^{n+1} \colon I^n).
\]
So $\widehat{I} = (I^{r+1} \colon I^r)$ for all $r \gg 0$ (see \cite{RR}).

Let $ G(I)  = \bigoplus_{n \geq 0}I^n/I^{n+1}$ be the associated graded ring of $I$. Fix an integer $r$ with $1 \leq r \leq d$. Then Shah proved that
if $\depth G(I) \geq r$  then $I^s = (I^s)_j $ for
$d + 1 - r \leq j  \leq d$, and for all $s \geq 1$; see \cite[Theorem 5]{S}. In particular if $G$ is \CM \ then $I^s = (I^s)_1$ for all $s \geq 1$.

To state our results we need to introduce some concepts. Let $R = \bigoplus_{n \geq 0}R_n$ be a standard graded algebra over a Artin local ring $(R_0, \m_0)$.
Let $M$ be a finitely generated  graded $R$-module. Let $H^i(M)$ denote the $i^{th}$-local cohomology module of $M$ \wrt \ $R_+$. It is well-known that
$H^{i}(M)$ are $*$-Artininian $R$-module (i.e., every descending chain of graded submodules stabilize). It follows that it's Matlis-dual $H^i(M)^\vee$ is a finitely generated
$R$-module. If $M$ is non-zero and of dimension $r$ then it is known that  $\dim H^i(M)^\vee \leq i$ for $0 \leq i \leq r-1$ and $\dim H^r(M)^\vee = r$; see
\cite[17.1.9 and 17.1.10]{BSh}.
We set dimension of the zero module to be $-1$.
In this paper we prove
\begin{theorem}\label{main}
 Let $(A,\m)$ be a \CM \ local ring of dimension $d \geq 2$ and with infinite residue field. Let $I$ be an $\m$-primary ideal of $A$.
 Fix an integer $r$ with $1 \leq r \leq d -1$. If $\dim H^i(G(I))^\vee \leq i -1$ for $1 \leq i \leq r$ then  $(I^n)_{d -r} = \widetilde{I^n}$ for all $n\geq 1$.
\end{theorem}
As an easy consequence we obtain
\begin{corollary}\label{gcm}
 (with hypotheses as in \ref{main}) If $G(I)$ is generalized \CM \ then $(I^n)_1 = \widetilde{I^n}$ for all $n \geq 1$.
\end{corollary}

In \cite{C} the first coefficient ideal $I$ is related to the $S_2$-ification of the Rees algebra. From his results it follows that if $A$ is an analytically unramified
\CM \ of dimension $d \geq 2$ then the $S_2$-ification of the Rees algebra $\R(I) = A[It]$
is  $\bigoplus_{n \geq 0}(I^n)_1 $.
 As a consequence we obtain
 \begin{corollary}\label{gcm-app}
 (with hypotheses as in \ref{gcm}) Further assume $A$ is analytically unramified domain.
 If $G(I)$ is generalized \CM \ then  $S_2$-ification of the Rees algebra $\R(I)$ is  $\bigoplus_{n \geq 0} \widetilde{I_n}$. In particular $\R(I^n)$ is $S_2$ for all $n \gg 0$.
\end{corollary}

 We now describe in brief the contents of this paper. In section two we introduce some notation and discuss some preliminary results that we need. In section three we discuss some properties of $L^I(M)$ that we need.
 In section four we discuss our results on dimensions of duals of certain local cohomology modules. In section five we prove Theorem \ref{main}.

\section{Notation and Preliminaries}
In this section we introduce some notation and discuss a few preliminaries which will
be used in this paper.
In this paper all rings are commutative Noetherian and all modules (\emph{unless stated otherwise})
are assumed finitely generated. We  use  terminology from  \cite{BH}.
Let
$(A,\m)$ be a local ring of dimension $d$ with residue field $k =
A/\m$. Let $M$ be \CM \  $A$-module of dimension $r$.  Throughout $I$ is an $\m$-primary ideal.

\s If $p \in M$ is non-zero and  $j$ is the largest integer such that $p \in I^{j}M$,
then we let $p^*$ denote the image of $p$
 in $I^{j}M/I^{j+1}M$.

\s
\label{hilbcoeff}
 The \emph{Hilbert function} of $M$ with respect to $I$ is the function
\[
 H_{I}(M,n) =  \lambda(I^nM/I^{n+1}M)\quad \text{for all} \ n \geq 0.
\]
It is well known that the formal power series $\sum_{n \geq 0}H_{I}(M,n)z^n$
represents a rational function of a
special type:
\begin{equation*}
\sum_{n \geq 0}H_{I}(M,n) z^n = \frac{h_{I}(M ,z)}{(1-z)^{r}}\quad \text{where}
\ r = \dim M \ \text{and} \ h_{I}(M,z) \in\mathbb{Z}[z].
\end{equation*}
 Set
$e_{i}^{I}(M) = (h_I(M,z))^{(i)}(1)/i! $ for all $i\geq 0$. The integers $e_{i}^{I}(M)$ are called \emph{Hilbert
coefficients} of $M$ with
respect to $I$.
The number $e_{0}^{I}(M)$ is also called the \emph{multiplicity} of $M$ with respect
to $I$.

\s Set $G(I) = \bigoplus_{n \geq 0}I^n/I^{n+1}$ to be the associated graded ring of $A$ \wrt \ $I$. If $M$ is an $A$-module then set
$G_I(M) = \bigoplus_{n \geq 0} I^nM/I^{n+1}M$ to be the associated graded module of $M$ \wrt \ $I$.

\s For definition and  basic properties of superficial sequences see \cite[p.\ 86-87]{Pu1}

\s\textbf{ Associated graded module and Hilbert function mod a superficial element: }\label{mod-sup-h} \\
Let $x \in I$
be $M$-superficial.  Set $N = M/xM$. There is a well-known relation between the Hilbert coefficients of $M$ and $N$. For
details   see cf., \cite[Corollary 10]{Pu1}.

\s \textbf{Base change:}
\label{AtoA'}
In our arguments we do use  a few base changes. See \cite[1.4]{Pu5}
for details.

We need the following result. It is definitely known to the experts. However we are unable to find a reference. We sketch a proof.
\begin{proposition}\label{dim-dual}
 Let $R = \bigoplus_{n \geq 0}R_n$ be a standard graded algebra over a complete Noetherian local ring $(R_0,\m_0)$. Let $M$ be a \emph{non-zero} finitely generated
 graded $R$-module of dimension $r$ with $\lambda(M_n) < \infty $ for all $n \in \Z$. Set $r = \dim M$. If $E$ is a graded $R$-module, set $E^\vee$ to be the Matlis dual
 of $E$ \wrt \ $R$. Let $\m = \m_0 \oplus R_+$ be the unique maximal homogeneous ideal of $R$. Then $\dim H^i_\m(M)^\vee \leq i$ for $i = 0,\ldots, r-1$ and
 $\dim H^r_\m(M)^\vee  = r$.
\end{proposition}
\begin{proof}[Sketch]
 The result is known if $R_0$ is Artin local \cite[17.1.9 and 17.1.10]{BSh}. Set $S = R/\ann_R M$ and let $\n$ be the maximal homogeneous ideal of $S$ . Then note that $S_0$ is Artin local
 and $H^i_\n(M) = H^i_\m(M)$ by graded independence theorem of local cohomology, \cite[13.1.6]{BSh}. It is also not difficult to show that the Matlis-dual of $H^i_I(M)$ \wrt \ $R$ is
 isomorphic to the Matlis dual of $H^i_\n(M)$ \wrt \ $S$. The result follows from the Artin local case.
\end{proof}

\section{Some Properties of $L^{I}(M)$}\label{Lprop}

In this section we collect some of  the properties of $L^{I}(M) = \bigoplus_{n\geq 0}M/I^{n+1}M $ which we proved in \cite{Pu5}. Throughout this section
$(A,\m)$ is a  local ring with infinite residue field, $M$ is a \emph{\CM }\ module of dimension $r \geq 1$ and $I$ an $\m$-primary ideal.

\s \label{mod-struc} Set $\Sc(I) = A[It]$;  the Rees Algebra of $I$. In \cite[4.2]{Pu5} we proved that \\
$L^{I}(M)$ is a $\Sc(I)$-module. Note that $L^I(M)$ is \emph{not finitely generated} $\R(I)$-module.

\s Let $H^{i}(-) = H^{i}_{\M}$ denote the $i^{th}$-local cohomology functor \wrt \ $\M$.
 Recall a graded $\Sc(I)$-module $V$ is said to be
\textit{*-Artinian} if
every descending chain of graded submodules of $V$ terminates. For example if $E$ is a finitely generated $\Sc(I)$-module then $H^{i}(E)$ is *-Artinian for all
$i \geq 0$.

\s \label{zero-lc} In \cite[4.7]{Pu5} we proved that
\[
H^{0}(L^I(M)) = \bigoplus_{n\geq 0} \frac{\wt{I^{n+1}M}}{I^{n+1}M}.
\]
Here $\widetilde{KM}$ denotes the Ratliff-Rush closure of $M$ \wrt \ an ideal $K$.  Recall
\[
\widetilde{KM} = \bigcup_{i \geq 1}K^{i+1}M \colon K^i.
\]
\s \label{Artin}
For $L^I(M)$ we proved that for $0 \leq i \leq  r - 1$
\begin{enumerate}[\rm (a)]
\item
$H^{i}(L^I(M))$ are  *-Artinian; see \cite[4.4]{Pu5}.
\item
$H^{i}(L^I(M))_n = 0$ for all $n \gg 0$; see \cite[1.10 ]{Pu5}.
\item
 $H^{i}(L^I(M))_n$  has finite length
for all $n \in \mathbb{Z}$; see \cite[6.4]{Pu5}.
\item
For $0 \leq i \leq r-1$
there exists polynomial $q_i(z) \in \mathbb{Q}[z]$ such that $q_i(n) = \ell(H^i(L^I(M))_n)$ for all $n \ll 0$.

\end{enumerate}

\s \label{II-FES} Let $x$ be  $M$-superficial \wrt \ $I$, i.e., $(I^{n+1}M \colon x) = I^nM$ for all $n \gg 0$. Set  $N = M/xM$ and $u =xt \in \Sc(I)_1$. Notice $L^I(M)/u L^I(M) = L^I(N)$.
For each $n \geq 1$ we have the following exact sequence of $A$-modules:
\begin{align*}
0 \xar \frac{I^{n+1}M\colon x}{I^nM} \xar \frac{M}{I^nM} &\xrightarrow{\psi_n} \frac{M}{I^{n+1}M} \xar \frac{N}{I^{n+1}N} \xar 0, \\
\text{where} \quad \psi_n(m + I^nM) &= xm + I^{n+1}M.
\end{align*}
This sequence induces the following  exact sequence of $\Sc$-modules:
\begin{equation}
\label{dagg}
0 \xar \B^{I}(x,M) \xar L^{I}(M)(-1)\xrightarrow{\Psi_u} L^{I}(M) \xrightarrow{\rho^x}  L^{I}(N)\xar 0,
\end{equation}
where $\Psi_u$ is left multiplication by $u$ and
\[
\B^{I}(x,M) = \bigoplus_{n \geq 0}\frac{(I^{n+1}M\colon_M x)}{I^nM}.
\]
We call (\ref{dagg}) the \emph{second fundamental exact sequence. }

\s \label{long-mod} Notice  $\lambda\left(\B^{I}(x,M) \right) < \infty$. A standard trick yields the following long exact sequence connecting
the local cohomology of $L^I(M)$ and
$L^I(N)$:
\begin{equation}
\label{longH}
\begin{split}
0 \xar \B^{I}(x,M) &\xar H^{0}(L^{I}(M))(-1) \xar H^{0}(L^{I}(M)) \xar H^{0}(L^{I}(N)) \\
                  &\xar H^{1}(L^{I}(M))(-1) \xar H^{1}(L^{I}(M)) \xar H^{1}(L^{I}(N)) \\
                 & \cdots \cdots \\
               \end{split}
\end{equation}

\s\label{Veronese} One huge advantage of considering $L^I(M)$ is that it behaves well \wrt \ the Veronese functor. Notice
\[
\left(L^I(M)(-1)\right)^{<t>} = L^{I^t}(M)(-1) \quad \text{for all} \ t \geq 1.
\]
Also note that $\Sc(I)^{<t>} = \Sc(I^t)$ and that $(\M_{\Sc(I)})^{<t>} = \M_{\Sc(I^t)}$. It follows that for all $i \geq 0$
\[
\left(H^i_{\M_{\Sc(I)}}(L^I(M)(-1)\right)^{<t>} \cong H^i_{\M_{\Sc(I^t)}}(L^{I^t}(M)(-1).
\]

\section{results on dimensions of dual's of local cohomology modules}
Throughout this section $(A,\m)$ is a complete Noetherian local ring with infinite
residue field and $M$ is a \CM \ $A$-module of dimension $r \geq 1$. Furthermore we will assume that $I$ is an $\m$-primary ideal.
In this section we prove some results regarding dimensions of dual's of graded local cohomology modules of $G_I(M)$ and $L^I(M)$. Throughout we compute local cohomology \wrt \ to the graded maximal ideal of the Rees algebra $\Sc(I) = A[It]$. If $E$ is a graded $\Sc$-module then we denote its Matlis dual by $E^\vee$.

\s Let $mod^f(\R(I))$ be the category of finitely generated graded $\R(I)$-modules $M$  with $\lambda(M)_n < \infty$ for all $n \in \Z$.
Let $A^f(\R(I))$ be the category of $*$-Artinian  graded $\R(I)$-modules $L$  with $\lambda(L)_n < \infty$ for all $n \in \Z$.
The usual Matlis duality between finitely generated graded $\Sc$-modules and $*$-Artinian $\R(I)$-modules (see \cite[3.6.17]{BH}) restricts to a duality between
$mod^f(\R(I))$ and $A^f(\R(I))$.

\s \label{filt-reg} Let $M \in mod^f(\R(I))$. Then it is not difficult to show that there exists a non-empty Zariski open subset $U$ of $I/\m I$ such  that if $x \in I$ such that $\ov{x} \in U$ then $\ker( M(-1)\xrightarrow{xt} M)$ has finite length. By Matlis duality we also obtain that
$\coker( M^\vee(-1)\xrightarrow{xt} M^\vee)$ has finite length.

\s \label{growth} Let $E \in A^f(\R(I))$. Then by Matlis duality it follows that there exists a polynomial $q_E(z)\in \mathbb{Q}[z]$ such that $q_E(n) = \lambda(E_n)$ for all $n \ll 0$. Furthermore if $E \neq 0$ then $\deg q_E(z) = \dim E^\vee - 1$. We call $q_E(z)$ as the \textit{dual Hilbert-polynomial} of $E$.

Our first result is
\begin{proposition}\label{order-L}
Let $M$ be a \CM \ $A$-module of dimension $r \geq 1$. Then $\dim H^i(L^I(M))^\vee
\leq i$ for $0 \leq i \leq r-1$.
\end{proposition}
\begin{proof}
We prove the result by induction on $r = \dim M$. If $r = 1$ then by \ref{zero-lc}
we get that $H^0(L^I(M))$ has finite length. So we have nothing to show.
Now assume $r \geq 2$ and the result has been shown for \CM \ modules of dimension $= r -1$. By \ref{filt-reg} we may choose $x \in I$ which is $M$-superficial and
$\coker H^i( L^I(M))(-1) \xrightarrow{xt} H^i( L^I(M)))$ has finite length for $i = 1,\ldots, r-1$.
Let $q_i^M(z)$ be the dual Hilbert-polynomial of $H^i( L^I(M))$.
By \ref{zero-lc}
we get that $H^0(L^I(M))$ has finite length. Set $N = M/xM$. Now assume that
$1 \leq i \leq r-1$. Then by our construction and \ref{long-mod} there exists finite length modules $U_i, V_i$ such that we have an exact sequence
\[
0 \rt U_i \rt H^{i-1}(L^I(N)) \rt H^{i}(L^I(M))(-1) \xrightarrow{xt} H^i( L^I(M))) \rt V_i \rt 0.
\]
By our induction hypothesis it follows that $\deg (q_i^M(z) - q_i^M(z-1)) \leq i-2$. So $\deg q_i^M(z) \leq i -1$. So $\dim H^i(L^I(M))^\vee
\leq i$.  The result follows.
\end{proof}

\begin{proposition}\label{o-G-less}
Let $M$ be a \CM \ $A$-module of dimension $r \geq 1$. Assume that there is $s$ with $  2 \leq s \leq r-1$ such that for  $1 \leq i  \leq s$ we have  \\ $\dim H^i(G_I(M))^\vee \leq i -1$. Then there exists a non-empty Zariski-open subset $W$ of $I/\m I$ such that if $x \in I$ and $\ov{x} \in W$ we have $\dim H^i(G_I(M/ x M))^\vee \leq i -1$ for $1 \leq i \leq  s-1$.
\end{proposition}
\begin{proof}
 By \ref{filt-reg}  it follows that we may choose
 a non-empty Zariski-open subset $W$ of $I/\m I$ such that if $x \in I$ and $\ov{x} \in U$ then $x $ which is
  $M$-superficial and
$\coker H^i( G_I(M))(-1) \xrightarrow{xt} H^i( G_I(M)))$ has finite length for $i = 1,\ldots, r-1$.
Set $N = M/x M$.
Then by our construction  there exists finite length modules $U_i, V_i$ such that we have an exact sequence for $2 \leq i \leq r-1$.
\[
 U_i \rt H^{i-1}(G_I(M)/xtG_I(M)) \rt H^{i}(G_I(M))(-1) \xrightarrow{xt} H^i( G_I(M))) \rt V_i .
\]
Using dual Hilbert polynomials
it follows that for $\dim H^i(G_I(M)/ xtG_I(M)))^\vee \leq i -1$ for $1 \leq i \leq  s-1$. It remains to observe that
$$ H^i(G_I(N))  \cong  H^i(G_I(M)/ xtG_I(M))  \quad \text{ for $i \geq 1$}.$$
\end{proof}

\begin{proposition}\label{o-L-less}
Let $M$ be a \CM \ $A$-module of dimension $r \geq 1$. Assume that there is $s$ with $  1 \leq s \leq r-1$ such that for  $1 \leq i  \leq s$ we have  \\ $\dim H^i(G_I(M))^\vee \leq i -1$. Then $\dim H^i(L^I(M))^\vee
\leq i-1$ for $1 \leq i \leq s$.
\end{proposition}
\begin{proof}
We do by induction on $s$. For $s = 1$ the result follows from \cite[5.2]{Pu6}.
Assume $s \geq 2$. Using  \ref{filt-reg} and \ref{o-G-less} it follows that
there exists a non-empty Zariski-open subset $W$ of $I/\m I$ such that if $x \in I$ and $\ov{x} \in W$ we have
\begin{enumerate}
\item
$x$ is $M$-superficial \wrt \ $I$.
\item
$\dim H^i(G_I(M/ x M))^\vee \leq i -1$ for $1 \leq i \leq  s-1$.
\item
$\coker H^i( L^I(M))(-1) \xrightarrow{xt} H^i( L^I(M)))$ has finite length for $i = 1,\ldots, r-1$.
\end{enumerate}
Let $q_i^M(z)$ be the dual Hilbert-polynomial of $H^i( L^I(M))$.
 Set $N = M/xM$. By induction hypothesis we have $\dim H^i(L^I(N))^\vee \leq i -1$ for $1 \leq i \leq s -1$.
 We prove $\dim H^i(L^I(M))^\vee \leq i -1$ for $1 \leq i \leq s$.  For $i = 1$ the result follows from \cite[5.2]{Pu6}.
 Now let $2 \leq i \leq s$.  By  our construction and \ref{long-mod} there exists finite length modules $U_i, V_i$ such that we have an exact sequence
\[
0 \rt U_i \rt H^{i-1}(L^I(N)) \rt H^{i}(L^I(M))(-1) \xrightarrow{xt} H^i( L^I(M))) \rt V_i \rt 0.
\]

By our induction hypothesis it follows that $\deg (q_i^M(z) - q_i^M(z-1)) \leq i-3$. So $\deg q_i^M(z) \leq i -2$. So $\dim H^i(L^I(M))^\vee
\leq i - 1$.  The result follows.
\end{proof}
\section{Proof of Theorem \ref{main}}
In this section we prove our main result.

\s \label{complete-red} For our arguments we have to go to the completion. The Ratliff-Rush closure  and integral closure of an $\m$-primary ideal behave well \wrt \ completion. However to the best of the authors knowledge it is not known whether other coefficient ideals behave well \wrt \ completion.
Set
$$E_i^\prime(I) = \{ J \mid  J \supset \widetilde{I} \ \text{and} \ e_j(J) = e_j(I) \  \   \text{for} \  0 \leq j \leq i \}. $$
It is clear that if $J \in  E_i^\prime(I)$ then $J\widehat{A} \in  E_i^\prime(I\widehat{A})$.
\s \label{setup-thm}
Fix $i$ with $1 \leq i \leq d$. Let $J \in E_i^\prime(I)$ then note $J \subseteq \ov{I}$. So $\R(J)$ is a finite $\R(I)$-module. Set $W(J) = \R(J)/\R(I)$. Note $W(J)_n$ has finite length for all $n$. We show
\begin{proposition}\label{dim-bound}
(with setup as in \ref{setup-thm}) $\dim W(J) \leq d - i$.
\end{proposition}
\begin{proof}
Set $W = W(J)$.
From the short exact sequence \[0 \rt W(+1) \rt \bigoplus_{n \geq 0} \frac{A}{I^{n+1}} \rt \bigoplus_{n \geq 0} \frac{A}{I_i^{n+1}} \rt 0\] we have
\begin{align*}
\lambda(W_{n+1})&= \lambda(A/I^{n+1})-\lambda(A/I_i^{n+1})\\
&= \{ e_0(I) \binom{n+d}{d} + \cdots +(-1)^je_j(I)\binom{n+d-j}{d-j}+ \cdots\}\\
& \quad -\{ e_0(J) \binom{n+d}{d}+ \cdots +(-1)^je_j(J)\binom{n+d-j}{d-j}+ \cdots\}
\end{align*}
 for all $n \gg 0$. As $e_j(I)=e_j(J)$ for all $0 \leq j \leq i$ so we get \[\lambda(W_{n+1})= (-1)^{i+1}\left(e_{i+1}(I)-e_{i+1}(J)\right) \binom{n+d-i-1}{d-i-1} + \mbox{ lower terms}.\]
 Thus $n \mapsto \lambda(W_{n+1})$ is a polynomial of degree at most $d-i-1$. Hence $\dim W \leq d-i$.
\end{proof}
 Theorem \ref{main} is an easy consequence of the following result.
 \begin{theorem}\label{main-L}
 Let $(A,\m)$ be a complete \CM \ local ring of dimension $d \geq 2$ and with infinite residue field. Let $I$ be an $\m$-primary ideal of $A$.
 Fix an integer $r$ with $1 \leq r \leq d -1$. Assume $\dim H^i(L^I(A))^\vee \leq i -1$ for $1 \leq i \leq r$. Then  for
 all $n \geq 1$ and $ E_{d-r}^\prime(I^n) = \{ \widetilde{I^n} \}$.
\end{theorem}
\begin{proof}
 We will first prove the result when $n = 1$. For the convenience of the reader we prove the case when $r = 1$. It suffices to show that
 $E_{d-1}^\prime(I) = \{ \widetilde{I} \}$. As discussed earlier we may assume that $A$ is complete.  Let $J \in E_{d-1}^\prime(I)$. Set $W = W(J) = \R(J)/\R(I)$.
 By \ref{dim-bound} we get $\dim W \leq 1$.
 We have an exact sequence
 \[
  0 \rt W(+1) \rt L^I(A) \rt L^J(A) \rt 0.
 \]
This induces a long exact sequence in cohomology
\[
\cdots \rt H^0(L^J(A)) \rt H^1(W(+1)) \rt H^1(L^I(A)) \rt \cdots.
\]
Taking Matlis duals we obtain an exact sequence
\[
 \cdots \rt H^1(L^I(A))^\vee \rt H^1(W(+1))^\vee \rt H^0(L^J(A))^\vee \rt \cdots.
\]
By our assumption and \ref{zero-lc} it follows that $\dim H^1(W(+1)) \leq 0$, As $\dim W \leq 1$
it follows from  \ref{dim-dual} that $W$ is zero-dimensional. Thus $J^m = I^m$ for all $m \gg 0$. Therefore $J \subseteq \widetilde{I}$. But by definition of $E_{d-1}^\prime(I)$
we have $J \supseteq \widetilde{I}$. Thus $ E_{d-1}^\prime(I) = \{ \widetilde{I} \}$.

Now assume $r \geq 2$. Let $J \in E_{d-r}^\prime(I)$. Set Set $W = W(J) = \R(J)/\R(I)$.
 By \ref{dim-bound} we get $\dim W \leq r$. We assert $\dim W = 0$. Suppose if possible $\dim W = c > 0$. Note $c \leq r$.
 We have an exact sequence
 \[
  0 \rt W(+1) \rt L^I(A) \rt L^J(A) \rt 0.
 \]
This induces a long exact sequence in cohomology
\[
\cdots \rt H^{c-1}(L^J(A)) \rt H^c(W(+1)) \rt H^c(L^I(A)) \rt \cdots.
\]
Taking Matlis duals we obtain an exact sequence
\[
 \cdots \rt H^c(L^I(A))^\vee \rt H^c(W(+1))^\vee \rt H^{c-1}(L^J(A))^\vee \rt \cdots.
\]
By our assumption and \ref{order-L} it follows that $\dim H^c(W(+1)) \leq c-1$. This contradicts \ref{dim-dual}. So $\dim W = 0$. Thus $J^m = I^m$ for all $m \gg 0$. Therefore $J \subseteq \widetilde{I}$. But by definition of $E_{d-r}^\prime(I)$
we have $J \supseteq \widetilde{I}$. Thus $ E_{d-r}^\prime(I) = \{ \widetilde{I} \}$.
Thus we have proved the result when $n = 1$.

Now assume that $n > 1$. We note that
$L^{I^n}(A)(-1)$ is the $n^{th}$-Veronese module of $L^I(A)(-1)$; see \ref{Veronese}. As local cohomology module commutes with Veronese it follows that
$\dim H^i(L^{I^n}(A)) \leq i - 1$ for $1 \leq i \leq r$. The result now follows from the $n = 1$ case.
\end{proof}

We now give
\begin{proof}[Proof of Theorem \ref{main}]
 It suffices to prove that $ E_{d-r}^\prime(I^n) = \{ \widetilde{I^n} \}$ for all $n \geq 1$. We note that $G(I) = G(I\widehat{A})$. By \ref{complete-red} we may assume
 that $A$ is complete. By \ref{o-L-less} it follows that $\dim H^i(L^I(A))^\vee \leq i - 1$ for $1 \leq i \leq r$. The result now follows from Theorem \ref{main-L}.
\end{proof}

\section*{Acknowledgements}
I thank Dr. Sudeshna Roy for help in typing this paper.


\begin{thebibliography}{10}


\bibitem{BSh}
M.P. Brodmann and R.Y. Sharp, \emph{{Local Cohomology: An algebraic introduction
  with geometric applications}}, vol.~60, Cambridge studies in advanced
  mathematics, Cambridge University Press,~Cambridge, 1998.

\bibitem{BH}
W.~Bruns and J.~Herzog, \emph{{Cohen-Macaulay rings}}, vol.~39, Cambridge
  studies in advanced mathematics, Cambridge University Press,~Cambridge, 1993.

\bibitem{C}
C.~Ciuperca,
\emph{First coefficient ideals and the $S_2$-ification of a Rees algebra},
J. Algebra 242(2001), 782--794.

\bibitem{Pu1}
T.~J. Puthenpurakal, \emph{{Hilbert coeffecients of a Cohen-Macaulay module}},
  J.~Algebra \textbf{264} (2003), 82--97.

\bibitem{Pu5} \bysame,
  \emph{Ratliff-{R}ush filtration, regularity and depth of higher associated graded modules. {I}},
   J. Pure Appl. Algebra \textbf{208} (2007), no.~1, 159--176.

   \bibitem{Pu6} \bysame,
  \emph{Ratliff-{R}ush filtration, regularity and depth of higher associated graded modules. {II}},
    J. Pure Appl. Algebra 221 (2017), no. 3, 611–-631.

    \bibitem{RR}
L.J. Ratliff and D.~Rush, \emph{{Two notes on reductions of ideals}}, Indiana
  Univ.~Math.~J \textbf{27} (1978), 929--934.

    \bibitem{Rees}
D.\ Rees,
 \emph{a-transform of local rings and a theorem on
multiplicities of ideals},
 Proc. Cambridge Philos Soc. (2) \textbf{57} (1961) 8-17.


\bibitem {S} K.~Shah,
\emph{Coefficient ideals},
Trans. Amer. Math. Soc.  327  (1991),  no. 1, 373--384.


\end{thebibliography}
\end{document}